\documentclass[12pt, english]{article}
\usepackage[utf8]{inputenc}
\usepackage{amsthm}
\usepackage{amssymb}
\usepackage[all]{xy}
\usepackage{amsmath}
\usepackage{tikz-cd}
\usepackage{tikz}
\usepackage{float}
\usepackage{bbm}

\tikzcdset{scale cd/.style={every label/.append style={scale=#1},
		cells={nodes={scale=#1}}}}
\usepackage{rotating}
\usepackage{caption, booktabs}
\usepackage{verbatim}
\usetikzlibrary{calc}

\usepackage{tikz}
\usetikzlibrary{patterns}

\usepackage{pdflscape}

\usepackage{mathtools}
\usepackage[english]{babel}

\usepackage{hyperref}
\usepackage{bookmark}

\newtheorem{definition}{Definition}[section]

\newtheorem{remark}[definition]{Remark}
\newtheorem{lemma}[definition]{Lemma}

\newtheorem{theorem}[definition]{Theorem}

\newtheorem{construction}[definition]{Construction}

\DeclareMathOperator{\Pic}{Pic}

\newcommand{\uproman}[1]{\uppercase\expandafter{\romannumeral#1}}

\begin{document}
	
	\title{Certain degenerations of surfaces in fake weighted projective $3$-space}
	\author{Julius Giesler,\\ University of T\"ubingen}
	\date{\today}
	\maketitle
	\begin{abstract}
		We study degenerations (along an edge) of surfaces in (fake) weighted projective $3$-space. We introduce the Hodge theoretic data and show that the Hodge loci of vanishing cycles both might be proper or not. This illustrates results of Green and Voisin.
	\end{abstract}

%criterion for the Noether-Lefschetz theorem for surfaces in $\mathbb{P}^3$ of degree $\geq 4$, the vanishing cohomology contain a nonzero $(2,0)$-cycle, extends to the weighted projective space, the second criterion of the Noether-Lefschetz theorem, that the Hodge loci of vanishing cycles are proper, fails for surfaces in (fake) weighted projective $3$-space This illustrates the most simplest cases of results of Voisin and Green.
\section{Introduction}
Let $\Delta$ be a $3$-dimensional lattice simplex and $f$ a nondegenerate Laurent polynomial with Newton polytope $\Delta$. Write $U_{reg}(\Delta)$ for all of these Laurent polynomials. Let $\mathbb{P}_{\Delta}$ be the projective toric variety to the normal fan of $\Delta$ and $Z_{\Delta,f}$ the closure of 
\begin{align*}
	Z_f := \{f=0\} \subset T
\end{align*}
in $\mathbb{P}_{\Delta}$. Then $\mathbb{P}_{\Delta}$ is a fake weighted projective $3$-space. Given a polytope $F$ let $l^*(F)$ denote the number of lattice points in the relative interior of $F$.

Assume that there is an edge $e \leq \Delta$ with $l^*(e) >0$. Then we may subdivide $\Delta$ along the edge $e$ into simplices as in the Figure below. 

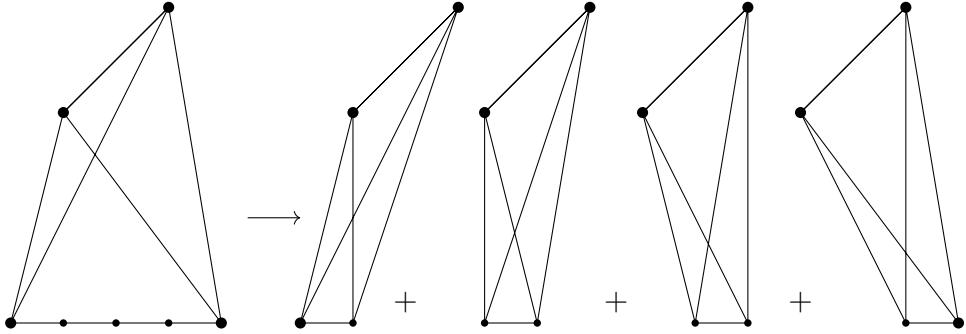
\begin{figure}[H] \label{figure_subdivision_along_edge_into_simplices}
	
	\begin{tikzpicture}[scale=0.7]

		\begin{scope}[scale = 1,yshift = 0cm, xshift = 2cm]
			
			\draw (0,0) -- (4,0) -- (1,4) -- (3,6) -- (0,0) -- (1,4) -- (3,6) -- (4,0);	
			
			\fill (0,0) circle(3pt);
			\fill (4,0) circle(3pt);
			\fill (1,4) circle(3pt);
			\fill (3,6) circle(3pt);
			
			\fill (1,0) circle(2pt);
			\fill (2,0) circle(2pt);
			\fill (3,0) circle(2pt);
			\fill (4,0) circle(2pt);
			
			\draw[->] (4.5,2)--(5.5,2);

		\end{scope}

		\begin{scope}[scale = 1, xshift =7.5cm]
			
			\draw (0,0) -- (1,0) -- (1,4) -- (3,6) -- (0,0) -- (1,4) -- (3,6) -- (1,0);	
			
			\fill (0,0) circle(3pt);
			\fill (1,0) circle(2pt);
			\fill (1,4) circle(3pt);
			\fill (3,6) circle(3pt);
			
			\fill (2,0) node[above] {+};
			
		\end{scope}

		\begin{scope}[scale = 1, xshift = 10cm]
			
			\draw (1,0) -- (2,0) -- (1,4) -- (3,6) -- (1,0) -- (1,4) -- (3,6) -- (2,0);	
			
			\fill (1,0) circle(2pt);
			\fill (2,0) circle(2pt);
			\fill (1,4) circle(3pt);
			\fill (3,6) circle(3pt);
			
			\fill (3.5,0) node[above] {+};

		\end{scope}
		
		\begin{scope}[scale = 1, xshift = 13cm]
			
			\draw (2,0) -- (3,0) -- (1,4) -- (3,6) -- (2,0) -- (1,4) -- (3,6) -- (3,0);	
			
			\fill (2,0) circle(2pt);
			\fill (3,0) circle(2pt);
			\fill (1,4) circle(3pt);
			\fill (3,6) circle(3pt);
			
			\fill (4,0) node[above] {+};

		\end{scope}
		
		\begin{scope}[scale = 1, xshift = 16cm]
			
			\draw (3,0) -- (4,0) -- (1,4) -- (3,6) -- (3,0) -- (1,4) -- (3,6) -- (4,0);	
			
			\fill (3,0) circle(2pt);
			\fill (4,0) circle(3pt);
			\fill (1,4) circle(3pt);
			\fill (3,6) circle(3pt);

		\end{scope}

	\end{tikzpicture}
	\caption{Subdivision of simplex $\Delta$ along an edge into simplices $\Delta_i$.}
\end{figure}

Such a subdivision corresponds to a degeneration both of the toric $3$-fold $\mathbb{P}_{\Delta}$ and of the hypersurface $Z_{\Delta,f}$. If $f \in U_{reg}(\Delta)$ is \glqq enough nondegenerate\grqq \,then we get
\[ \dim \, H^{2,0}(Z_{\Delta,f}) = \underbrace{p_g(Z_{\Delta,f})}_{l^*(\Delta)} = \sum\limits_{i = 1}^r  \underbrace{p_g(Z_{\Delta_i,f_i})}_{l^*(\Delta_i)} + \sum\limits_{i = 1}^{r-1} \underbrace{g(Z_{\Delta_i,f_i} \cap Z_{\Delta_{i+1},f_{i+1}})}_{l^*(\Delta_i \cap \Delta_{i+1})}  \]
with $Z_{\Delta_i,f_i}$ a closure in $\mathbb{P}_{\Delta_i}$. We restrict to a one-dimensional subfamily $Z_{\Delta,t}$, $t \in B$, where $B$ denotes the unit disc, and switch to a semistable degeneration $\mathcal{X} \overset{p}{\rightarrow} B$ by base change and birational maps (see construction \ref{construction_semistable_degeneration}). \\ \\
Write $X_t := p^{-1}(t)$. Then given $t \neq 0$ the fibre $X_t$ is a resolution of singularities of $Z_{\Delta,t}$ and there is a toric morphism $\mathbb{P}' \rightarrow \mathbb{P}_{\Delta}$ inducing this resolution. Considering the \textit{geometric genus} $p_g(X_t)$ we can keep track of the original degeneration within this semistable degeneration.
% All additional Picard classes on $X_t$ (compared to $Z_{\Delta,t}$) come from the surrounding toric variety $\mathbb{P}'$. \\
There are cohomology classes on $X_t:= p^{-1}(t)$, $t \neq 0$, that come from $X_0$. But there might be also additional \textit{vanishing cohomology classes}. These vanishing cohomology classes are rational cohomology classes and orthogonal on those coming from $X_0$ with respect to the intersection product on $H^2(X_t, \mathbb{Q})$. \\ \\
The number of these vanishing cohomology classes in the semistable degeneration from above equals
\[ 2 \cdot \sum\limits_{i=1}^{r-1} l^*(\Delta_{i} \cap \Delta_{i+1}).  \]
Since these classes are rational and $H^{0,2}(X_t) = \overline{H^{2,0}(X_t)}$, assuming
\begin{align*}
	\sum\limits_{i=1}^{r-1} l^*(\Delta_{i} \cap \Delta_{i+1}) > \sum\limits_{i = 1}^r l^*(\Delta_i),
\end{align*}
there is a cohomology class $\lambda_t \in H_{van}^2(X_t,\mathbb{C}) \cap H^{1,1}(X_t)$. This means that the Hodge locus $U_{\lambda}^1 := \{ f \in U_{reg}(\Delta) \mid \lambda_f \in F^1 H^2(Z_{\Delta,f},\mathbb{C}) \}$ is not proper:

\begin{theorem}
	Let $\Delta$ be a $3$-dimensional simplex having an edge $e$ with $l^*(e) >0$. Let $f \in U_{reg}(\Delta)$ be sufficiently nondegenerate and consider the subdivision
	\[ \Delta = \Delta_1 \cup ... \cup \Delta_r, \quad \Delta_{i,i+1}:= \Delta_i \cap \Delta_{i+1}  \]
	from above. If 
	\begin{align} \label{assumption_interior_points_intersection_subdivision}
		\sum\limits_{i=1}^{r-1} l^*(\Delta_{i,i+1}) > \sum\limits_{i = 1}^r l^*(\Delta_i),
	\end{align}
	there is $\lambda \in H_{van}^2(X_t,\mathbb{C})$ such that the Hodge locus $U_{\lambda}^1$ is not a \glq proper\grq \, subvariety.
\end{theorem}

The criterion (\ref{assumption_interior_points_intersection_subdivision}) appears already for surfaces in weighted projective $3$-space, as one can check with a computer program.  In the following we check (Construction \ref{construction_criterion_of_Voisin}) that a \textit{Voisin type criterion}
\begin{align*}
	H^{2,0}(X_t) \cap H_{van}^2(X_t,\mathbb{C}) \neq \{0\}
\end{align*}
remains valid (except if $H_{van}^2(X_t,\mathbb{C}) = \{0\}$, which is not essential, see Remark \ref{remark_finitely_many_times_Voisin_type}) in our setting of surfaces in (fake) weighted projective $3$-space. This criterion might be helpful in proving the \textit{properness} of the \textit{Noether-Lefschetz locus}.
% We guess that this criterion applies to the situation to give the \textit{properness} of the Noether-Lefschetz locus for surfaces in (fake) weighted projective $3$-space with $p_g(X_t)> 0$ (and degree $k \cdot q$, where $q=\lcm(q_0,...,q_3)$, $k \geq 1$, for weighted projective space $\mathbb{P}(q_0,...,q_3)$ and similarly for fake weighted projective spaces), the same as for surfaces $X \subset \mathbb{P}^3$ of degree $\geq 4$. 
\\ \\
This paper studies elementary degenerations (degenerations along an edge) and transports two approaches to the \textit{properness} of the \textit{Noether-Lefschetz locus} of surfaces in $\mathbb{P}^3$ (compare \cite[Thm.3.33, section 6.3]{Voi03}) to the setting of surfaces in (fake) weighted projective $3$-space. It serves as an elementary exemplification of a result of Green (\cite{Gre84},\cite[Thm.6.28]{Voi03}). In the first version of this preprint there was an elementary mistake. We note that correcting this mistake gives our result a different intuition.

\section{Degenerations and the geometric genus}

\begin{definition}
	A degeneration of algebraic surfaces is a proper flat morphism $p: \mathcal{X} \rightarrow B$, where $B \subset \mathbb{C}$ is the unit disc, such that $X_t:= p^{-1}(t)$ is a smooth projective surface for $t \neq 0$ and the total space $\mathcal{X}$ is a K\"ahler manifold. The degeneration is called semistable if we may write $X_0 = \sum\limits_{i = 1}^r V_i$, where $V_i$ are smooth and the irreducible components of $X_0$, the $V_i$ intersect transversally and locally $p$ is defined by 
	\[ t = x_1 \cdot ... \cdot x_k   \]
\end{definition}

\begin{construction} (Subdivison of a simplex along an edge) \\
	\normalfont
	Assume that $\Delta$ is a $3$-dimensional simplex with $l^*(\Delta) > 0$ and having an edge $e$ with $r:= l^*(e) > 0$. Choose a subdivision 
	\begin{align*}
		& \Delta = \Delta_1 \cup ... \cup \Delta_r \\
		& \Delta_{i,i+1} := \Delta_i \cap \Delta_{i+1}.
	\end{align*}
	
	Let $f_t \in U_{reg}(\Delta)$ be a one-parameter family of nondegenerate Laurent polynomials. Then we may choose a birational toric morphism $\mathbb{P}' \rightarrow \mathbb{P}_{\Delta}$ such that the closure $X_t:= Z_t'$ of $Z_{f_t}$ is smooth for $t \neq 0$.
	\\
	Assume that $f_{0 \vert{\Delta_i}} \in U_{reg}(\Delta_i)$ for $i=1,...,r$ and $f_{0 \vert{\Delta_{i,i+1}}} \in U_{reg}(\Delta_{i,i+1})$ for $i=1,..,r-1$. Denoting by $Z_{\Delta_i}$ the closure of $\{f_{0 \vert{\Delta_i}} = 0\}$ in $\mathbb{P}_{\Delta_i}$ we may assume that $Z_{\Delta_i}$ and $Z_{\Delta_{i+1}}$ intersect transversally. Since $Z_{\Delta_i} \subset \mathbb{P}_{\Delta_i}$ is ample this is a nonempty Zariski open conditions on the coeffficients of $f_0$.
\end{construction}

\begin{construction} \label{construction_semistable_degeneration} (Semistable degeneration) \\
	\normalfont
	Given a degeneration $p: \mathcal{X} \rightarrow B$ there exists a base change $\phi: B \rightarrow B$ given by  $t \mapsto t^N$, a semistable degeneration $p': \mathcal{Y} \rightarrow B$ and a diagram
	\begin{equation*}
		\begin{tikzcd}
			\mathcal{Y} \arrow[dashed]{r}{\psi} \arrow{dr}{p'} & \mathcal{X} \times_{\phi} B \arrow[d] \arrow[r] & \mathcal{X} \arrow{d}{p} \\
			& B \arrow{r}{\phi} & B
		\end{tikzcd}	
	\end{equation*}
	such that $\psi$ is a birational map of the central fibre given by blowing up and blowing down subvarieties. 
\end{construction}

\begin{lemma}
	Let $\Delta$ be a $3$-dimensional lattice simplex with $l^*(\Delta) >0 $ and an edge $e$ wit $l^*(e) >0$. Choose the subdivision 
	\[ \Delta = \Delta_1 \cup ... \cup \Delta_r, \quad \Delta_{i,i+1} := \Delta_i \cap \Delta_{i+1} \]
	along $e$ with hypersurfaces $Z_{\Delta}$, $V_i := Z_{\Delta_i}$ and double curves $C_{i,i+1} := Z_{\Delta_i} \cap Z_{\Delta_{i+1}}$. Then
	\[ p_g(X_t) = \sum\limits_{i = 1}^r p_g(Z_{\Delta_i}) + \sum\limits_{i = 1}^{r-1} g(C_{i,i+1}).  \]
	As a consequence for all additional components $V_{r+1},...,V_l$ in a semistable degeneration and all other double curves $C_{ij}$ we have
	\[ p_g(V_i) = 0, \quad g(C_{ij}) = 0. \]
\end{lemma}

\begin{proof}
	We have $p_g(X_t) = l^*(\Delta)$ and $p_g(Z_{\Delta_i}) = l^*(\Delta_i)$ for $i=1,...,r$ by \cite{DK86}. Further $g(C_{i,i+1}) = l^*(\Delta_{i,i+1})$ and
	\begin{align} \label{geom_genus_additive_sum}
		l^*(\Delta) = \sum\limits_{i = 1}^r l^*(\Delta_i) + \sum\limits_{i = 1}^{r-1} l^*(\Delta_{i,i+1}).
	\end{align}	
	Let $\mathcal{X} \rightarrow B$ be a semistable degeneration one obtains by the semistable reduction Theorem. Then by (\cite[Prop.2.7.3]{Pers77}) we have
	\begin{align}\label{formula_geometric_genus_degeneration}
		p_g(X_t) = \sum\limits_{i = 1}^l p_g(V_i) + \sum\limits_{i \neq j} g(C_{ij}) + h^2(\Gamma). 	
	\end{align}
	where $\Gamma$ denotes the dual graph of $X_0$. But all nonzero terms are already contained in $p_g(Z_{\Delta_i})$ and $g(C_{i,i+1})$. Thus $h^2(\Gamma) = 0$ and the result follows.
\end{proof}

In the following we summarize some known results on degenerations (of algebraic surfaces). See (\cite[Ch.2]{Pers77}) and (\cite{Mor84}) for details.

\begin{remark} \label{remark_diff_subdiv}
	\normalfont
	The result stays true for different subdivisions of $\Delta$ such that no interior point of $\Delta$ belongs to an edge of the subdivision.	
Let 
\[ T: H^2(X_t, \mathbb{Q}) \rightarrow H^2(X_t, \mathbb{Q}), \quad t \neq 0  \]
be the monodromy homomorphism and $N := \log(T)$. Then in the above situation in fact $N = T-I$ since $h^2(\Gamma) =0$. 
\end{remark}

\begin{construction} (Clemens-Schmid exact sequence) \\
	\normalfont
	By the \textit{invariant cycle theorem} (\cite[Thm.4.24]{Voi03}) the $T$-invariant cohomology classes of $H^2(X_t, \mathbb{Q})$ come from
	$H^2(X_0, \mathbb{Q})$, that is	the following sequence is exact at the middle term
	\begin{align} \label{exact_sequence_part_Clemens_Schmid}
	H^2(X_0, \mathbb{Q}) \overset{\alpha}{\rightarrow} H^2(X_t, \mathbb{Q}) \overset{N}{\rightarrow} H^2(X_t, \mathbb{Q}).	
	\end{align}
	In fact this triple fits into an exact sequence, the \textit{Clemens-Schmid exact sequence} (see \cite[Ch.2.7 Thm.V]{Pers77} for details on the maps $\alpha, \, \beta$ and $\gamma$)
	\begin{align} \label{Clemens_Schmidt_exact_sequence}
		0 &\rightarrow H^{0}(X_t, \mathbb{Q}) \rightarrow H_{4}(X_0, \mathbb{Q}) \overset{\alpha}{\rightarrow} H^{2}(X_0, \mathbb{Q}) \overset{\beta}{\rightarrow} H^2(X_t, \mathbb{Q}) \\
		& \overset{N}{\rightarrow} H^{2}(X_t, \mathbb{Q})  \overset{\gamma}{\rightarrow} H_2(X_0, \mathbb{Q})  \nonumber
	\end{align}
\end{construction}

\begin{definition}
	We define 
	\[ H^2(X_t, \mathbb{Q})_{van} \cong H^2(X_t, \mathbb{Q})/\ker(N).  \]
%	\[ H^2(X_t, \mathbb{Q})_{van} \cong \ker(\gamma: H^2(X_t, \mathbb{Q}) \rightarrow H_2(X_0, \mathbb{Q})). \]
	and call its elements \textit{vanishing cohomology classes}.
\end{definition}

%\begin{construction} (Limiting mixed Hodge structure) \\
%	\normalfont
%	By the Clemens-Schmid exact sequence
%	\[ H^2(X_t, \mathbb{Q})_{van} \cong H^2(X_t, \mathbb{Q})/\ker(N).  \]

\begin{construction} \label{construction_vanishing_cohomology} (Limiting mixed Hodge structure) \\
	\normalfont
	We could equip $X_t$ and $X_0$ with mixed Hodge structures and then the Clemens-Schmidt exact sequence gets an exact sequence of MHS: Let
	\[ 0 \subset W_0 \subset W_1 \subset W_2 \subset W_3 = W_4 = H^2(X_t, \mathbb{Q})  \]
	be the weight filtration. Then $W_0=0$ since $h^2(\Gamma) = 0$ and since $q(V_i)=0$ we get
\begin{align*}
	&W_2/W_1 = \ker \Big( \bigoplus\limits_i H^{2}(V_i,\mathbb{C}) \rightarrow \bigoplus\limits_i H^2(C_{i,i+1},\mathbb{C}) \Big), \\
	&W_1 =  W_1/W_0  \cong \bigoplus\limits_{i = 1}^{r-1} H^1(C_{i,i+1}, \mathbb{Q})
\end{align*} 
	Further $N$ is of type $(-1,-1)$ thus $H^2(X_t, \mathbb{Q})_{van} = W_3/W_2 \cong W_1/W_0$. We write $F_{lim}^p$ for the \textit{limiting Hodge structure} on $H^2(X_t,\mathbb{C})$. The mixed Hodge structure on $H^2(X_0,\mathbb{C})$ is computed via \textit{Mayer-Vietoris} sequences from equi-dimensional strata of $X_0$.
\end{construction}

\begin{remark} (Vanishing classes orthogonal on $\ker(N)$) \\
	\normalfont
	The importance of vanishing cohomology classes arise from the fact that we have a orthogonal decomposition
	\begin{align} \label{orthogonal_sum_decomposition_van_coh} 
		H^2(X_t, \mathbb{Q}) \cong H^2(X_t, \mathbb{Q})_{van} \oplus \ker(N). 
	\end{align}
	with respect to the intersection pairing. Following (\cite[Proof of Prop.2.27]{Voi03}) we give a proof of the orthogonality of this direct sum: $X_0$ is a deformation retract of $\mathcal{X}$. By Poincaré duality we get
	\[  H_2(X_0, \mathbb{Q}) \cong  H_2(\mathcal{X}, \mathbb{Q}) \cong H^4(\mathcal{X}, \mathbb{Q}) \]
	and the homomorphism $\gamma$ may be identified with the Gysin homomorphism $j_{*}: H^2(X_t, \mathbb{Q}) \rightarrow H^4(\mathcal{X}, \mathbb{Q})$, $j: X_t \rightarrow \mathcal{X}$ the inclusion. Thus we get $H^2(X_t, \mathbb{Q})_{van} \cong \ker(j_*)$. It follows
	\[ x.j^{*}(y) = j_*(x).y = 0 \]
	for $x \in H^2(X_t, \mathbb{Q})_{van}$, $y \in H^2(\mathcal{X}, \mathbb{Q})$ by the projection formula. Since $\ker(N) \cong j^{*}H^2(\mathcal{X}, \mathbb{Q})$ the direct sum is orthogonal.
	
\end{remark}

	%Further $X_0$ is a deformation retract of $\mathcal{X}$ and $\iota^* H^2(\mathcal{X}, \mathbb{Q}) \subset \ker(N)$ by the exact sequence \ref{exact_sequence_part_Clemens_Schmid} where $\iota: X_t \rightarrow \mathcal{X}$ denotes the inclusion.

\section{Vanishing cohomology classes and the Hodge decomposition}

\begin{theorem}\label{theorem_properness_Hodge_loci}
	Let $\Delta$ be a $3$-dimensional simplex having an edge $e$ with $l^*(e) >0$. Let $f \in U_{reg}(\Delta)$ be sufficiently nondegenerate and consider the subdivision
	\[ \Delta = \Delta_1 \cup ... \cup \Delta_r, \quad \Delta_{i,i+1}:= \Delta_i \cap \Delta_{i+1}  \]
	from above. If 
	\begin{align} \label{assumption_interior_points_intersection_subdivision}
		\sum\limits_{i=1}^{r-1} l^*(\Delta_{i,i+1}) > \sum\limits_{i = 1}^r l^*(\Delta_i),
	\end{align}
there is $\lambda \in H_{van}^2(X_t,\mathbb{C})$ such that the Hodge locus $U_{\lambda}^1 := \{ f \in U_{reg}(\Delta) \mid \lambda_f \in F^1 H^2(Z_{\Delta,f},\mathbb{C}) \}$ is not a \glq proper\grq \, subvariety.
\end{theorem}

\begin{proof}
	The vanishing cohomology classes lie in $H^2(X_t, \mathbb{Q})$ and the space $H^{2,0}(X_t) \oplus H^{0,2}(X_t)$ intersects $H^2(X_t, \mathbb{Q})$ in a $\leq p_g$-dimensional vector subspace. Thus by \ref{assumption_interior_points_intersection_subdivision} and construction \ref{construction_vanishing_cohomology} there is a linear combination
	\begin{align*}
		\lambda := \sum\limits_{i} a_i \cdot V_i, \quad a_i \in \mathbb{R}, \quad \textrm{with } \pi\big(\lambda \big) = 0,
	\end{align*}
	where $\pi: H^2(X_t,\mathbb{C}) \rightarrow H^{2,0}(X_t) \oplus H^{1,1}(X_t)$ denotes the natural projection. Thus $\lambda \in H_{van}^2(X_t, \mathbb{R}) \cap H^{1,1}(X_t)$ and the \textit{Hodge locus} $U_{\lambda}^1$ is not proper.
\end{proof}

\begin{construction} \label{construction_criterion_of_Voisin}
	\normalfont
	Unlike the criterion (for the \textit{properness} on the Noether-Lefschetz locus) via the \textit{Hodge loci} we guess that another criterion of Voisin (\cite[Thm.3.33]{Voi03}) might be helpful for generalizations to surfaces in (fake) weighted projective $3$-space:
	\begin{itemize}
		\item Either $H_{van}^2(X_t,\mathbb{C}) = \{0\}$ or $H_{lim}^{2,0}(X_t) \cap H_{van}^2(X_t,\mathbb{C}) \neq \{0\}$.
		\item $H_{lim}^{2,0}(X_t) = L^*(\Delta) = H^{2,0}(X_t)$ 
	\end{itemize}
%, $T$ is trivial, $H^2(X_t,\mathbb{Q}) \cong H^2(X_0,\mathbb{Q})$ and the computation of Picard classes $H^{1,1}(X_t) \cap H^2(X_t,\mathbb{Q})$ is reduced to surfaces $X_{0,i}$ with $l(\Delta_i) < l(\Delta)$. \\
This gives us a \textit{Voisin type criterion}
\begin{align} \label{Voisin's_criterion_properness_NL}
	H^{2,0}(X_t) \cap H_{van}^2(X_t,\mathbb{C}) \neq \{0\}.
\end{align}
%Note that $H^4(\mathcal{X}, \mathbb{C}) \cong H^4(\mathbb{P}_{\Delta},\mathbb{C})$ by Lefschetz hyperplane theorem for the $3$-dimensional space $\mathcal{X}$, thus the vanishing cohomology that we use is the same as the one defined by Vosin.
\begin{itemize}
		\item Voisin's proof of the \textit{properness} of the \textit{Noether-Lefschetz locus} (compare also the conjecture in \cite{JSt91}) works with \textit{Lefschetz pencils} (and the respective vanishing cohomology) and uses the irreducibility of the action of $T$ on the vanishing cohomology classes. We therefore could not finish a proof.
		 
		%should generalize to \textit{quasismooth} hypersurfaces $Z_{\Delta,f}$ (or by birational geometry to $X_t \subset \mathbb{P}$ only nef and big instead of very ample).
	\end{itemize}
\begin{proof} \textit{(of the first two points)} \\	
	If $\l^*(\Delta_{i,i+1}) = 0$ for all $i=1,...,r-1$ then $H_{van}^2(X_t,\mathbb{C}) = \{0\}$ and $T=I$ is trivial. Else we read off the \textit{limiting Hodge filtraton} somehow involved from the Clemens-Schmid exact sequence and the Hodge filtration of $H^2(X_0,\mathbb{C})$ computed via equidimensional strata. We claim that
	\begin{align} \label{equation_lim_hodge_str_20_intersection}
		H_{lim}^{2,0}(X_t) \cong \Big(H_{lim}^{2,0}(X_t) \cap H_{van}^2(X_t,\mathbb{C}) \Big) \oplus \Big( H_{lim}^{2,0}(X_t) \cap \ker(N) \Big)
	\end{align}	
	The inclusion $\supset$ is clear. With the notation of Construction \ref{construction_vanishing_cohomology} we have
	\begin{align*}
		& \ker(N) = W_2 = W_1  \oplus W_2/W_1 , \\
		& \Rightarrow H_{lim}^{2,0}(X_t) \cap \ker(N) = F_{lim}^2 W_2/W_1  \cong \bigoplus\limits_{i} H^{2,0}(V_i).
	\end{align*}
	and 
	\begin{align*}
		& H_{van}^{2}(X_t,\mathbb{C}) \cong W_3/W_2 \cong W_1/W_0 \cong \bigoplus\limits_i H^1(C_{i,i+1},\mathbb{C}), \\
		& H_{lim}^{2,0}(X_t) \cap H_{van}^{2}(X_t,\mathbb{C}) = F_{lim}^2 W_3/W_2 \cong \bigoplus\limits_i H^{1,0}(C_{i,i+1}).
	\end{align*}
	Thus the equation \ref{equation_lim_hodge_str_20_intersection} follows and since $\dim \, \ker(N) < h^{2,0}(X_t)$ the first point follows. Besides by the formula for the geometric genus naturally $H^{2,0}(V_i) \cong L^*(\Delta_i)$ and $H^{1,0}(C_{i,i+1}) \cong L^*(\Delta_{i,i+1})$, and the second point follows  
	\begin{align*}
		L^*(\Delta) = \bigoplus\limits_i L^*(\Delta_i) \oplus \bigoplus\limits_i L^*(\Delta_{i,i+1}).
	\end{align*}
	
\end{proof}
\end{construction}

\begin{comment}
\begin{remark}
	\normalfont
	We summarize some useful properties:
	\begin{align*}
		& T \textrm{ is quasi-unipotent} \\
		& \{X_t\}_t \textrm{ semistable } \Rightarrow T \textrm{ unipotent} \\
		& \textrm{ Our degeneration along an edge: } h^2(\Gamma) = 0, \, N = T-I, \, q(V_i) = 0, \\
		& \quad \quad \quad \quad \quad \quad \quad \quad \quad \quad \quad \quad \quad \quad \quad H_{van}^2(X_t,\mathbb{C}) \cong \bigoplus\limits_i H^1(C_{i,i+1}, \mathbb{C}) \\
		& \textrm{ Lefschetz-pencil: The monodromy } T \textrm{ acts irreducible on } H_{van}^2(X_t,\mathbb{C}) 
	\end{align*}
\end{remark}
\end{comment}

\begin{remark}\label{remark_finitely_many_times_Voisin_type}
	\normalfont
	If there is not necessarily an edge $e \leq \Delta$ with $l^*(e) > 0$ we can still subdivide $\Delta$ and find a degeneration with possibly $h^2(\Gamma) > 0$. In this case we get $W_0 = H^2(\left| \Gamma \right|, \mathbb{C})$ and $\ker(N) = W_2/W_0$, $H_{van}^2(X_t,\mathbb{C}) \cong W_1/W_0 \oplus W_0$. Thus $H_{van}^2(X_t,\mathbb{C})$ just can get bigger in this case. The proofs of the first is valid in this cases as well. Thus we definitely find a degeneration $\mathcal{X} \rightarrow B$ of $X_t$ with $H_{van}^2(X_t,\mathbb{C}) \neq \{0\}$ and we obtain a criterion of \textit{Voisin type}: $H^{2,0}(X_t) \cap H_{van}(X_t,\mathbb{C}) \neq \{0\}$. 
\end{remark}

\begin{remark}
	\normalfont
%	Thus while we guess that the first criterion (\ref{Voisin's_criterion_properness_NL}) of Voisin applies to (fake) weighted-projective spaces, the second criterion (this concludes from the properness of the Hodge loci in particular the properness of the \textit{integral} classes in the Hodge locus) giving a proof of the \textit{properness} of the Noether-Lefschetz locus (of surfaces of degree $\geq 4$) in $\mathbb{P}^3$ fails already for weighted projective $3$-spaces, namely if condition (\ref{assumption_interior_points_intersection_subdivision}) is true. 
By rersults of Green (\cite{Gre84},\cite[Thm.6.28]{Voi03}) Theorem (\ref{theorem_properness_Hodge_loci}) cannot happen if $Z_{\Delta,f}$ is sufficiently ample, for example when replacing $\Delta$ by $k \cdot \Delta$ for some natural number $k \gg 0$, not violating our condition \ref{assumption_interior_points_intersection_subdivision}. 
\end{remark}

\end{document}